\documentclass{amsart}
\usepackage{amssymb}
\usepackage{verbatim}
\usepackage{array}
\usepackage{latexsym}
\usepackage{enumerate}
\usepackage{amsmath}
\usepackage{amsfonts}
\usepackage{amsthm}
\usepackage{color}
\usepackage[english]{babel}
\usepackage{cancel}

\newtheorem{theorem}{Theorem}[section]
\newtheorem{proposition}[theorem]{Proposition}
\newtheorem{lemma}[theorem]{Lemma}
\newtheorem{corollary}[theorem]{Corollary}

\newtheorem{remark}[theorem]{Remark}

\newtheorem*{definition*}{Definition}

\newtheorem*{proposition*}{Proposition}
\newtheorem*{corollary*}{Corollary}
\newtheorem*{lemma*}{Lemma}
\newtheorem*{theorem*}{Theorem}

\def\cW{\mathcal W}
\def\cX{\mathcal X}
\def\cY{\mathcal Y}
\def\cZ{\mathcal Z}

\def\deg{\mbox{\rm deg}}

\def\bq{{\bar q}}

\def\gg{\mathfrak{g}}

\newcommand{\aut}{\mbox{\rm Aut}}

\newcommand{\K}{\mathbb{K}}

\title{Generalized Artin-Mumford curves over finite fields}
\date{}
\author{Maria Montanucci and Giovanni Zini}
\begin{document}




\begin{abstract}
Let $\mathbb{F}_q$ be the finite field of order $q=p^h$ with $p>2$ prime and $h>1$, and let $\mathbb{F}_{\bar{q}}$ be a subfield of $\mathbb{F}_q$.
From any two $\bar{q}$-linearized polynomials $L_1,L_2 \in \overline{\mathbb{F}}_q[T]$ of degree $q$, we construct an ordinary curve  $\mathcal{X}_{(L_1,L_2)}$ of genus $\gg=(q-1)^2$ which is a generalized Artin-Schreier cover of the projective line $\mathbb{P}^1$. The  automorphism group of $\mathcal{X}_{(L_1,L_2)}$ over the algebraic closure $\overline{\mathbb{F}}_q$ of $\mathbb{F}_q$ contains a semidirect product $\Sigma \rtimes \Gamma$ of an elementary abelian $p$-group $\Sigma$ of order $q^2$ by a cyclic group $\Gamma$ of order $\bar{q}-1$. We show that for $L_1 \neq L_2$, $\Sigma \rtimes \Gamma$ is the full automorphism group $\aut(\mathcal{X}_{(L_1,L_2)})$ over $\overline{\mathbb{F}}_q$; for $L_1=L_2$ there exists an extra involution and $\aut(\mathcal{X}_{(L_1,L_1)})=\Sigma \rtimes \Delta$ with a dihedral group $\Delta$ of order $2(\bar{q}-1)$ containing $\Gamma$. Two different choices of the pair $\{L_1,L_2\}$ may produce birationally isomorphic curves, even for $L_1=L_2$. We prove that any curve of genus $(q-1)^2$ whose $\overline{\mathbb{F}}_q$-automorphism group contains an elementary abelian subgroup of order $q^2$ is birationally equivalent to $\mathcal{X}_{(L_1, L_2)}$ for some separable $\bar{q}$-linearized polynomials $L_1,L_2$ of degree $q$. We produce an analogous characterization in the special case $L_1=L_2$. This extends a result on the Artin-Mumford curves, due to Arakelian and Korchm\'aros \cite{AK}. 
\end{abstract}

\maketitle

\section{Introduction}

The Artin-Mumford curve $\mathcal{M}_c$ of genus $(p-1)^2$ defined over a field $\mathbb{F}$ of odd characteristic $p$ is the nonsingular model of the plane curve with affine equation
\begin{equation}
\label{AM}
(X^p-X)(Y^p-Y)=c, \quad c \in \mathbb{F}^*.
\end{equation}
Artin-Mumford curves, especially over non-Archimedean valued fields of positive characteristic, have been investigated in several papers; see \cite{CK}, \cite{CKK}, and \cite{CK2}. By a result of Cornelissen, Kato and Kontogeorgis \cite{CKK} valid over any non-Archimedean valued field $(\mathbb{F}, \mid \cdot \mid)$ of positive characteristic, if $| c | <1$ then  $\aut_{\mathbb{F}}(\mathcal{M}_c)$ is the semidirect product
\begin{equation}
\label{Konto}
(C_p \times C_p) \rtimes D_{p-1},
\end{equation}
where $C_p$ is a cyclic group of order $p$ and $D_{p-1}$ is a dihedral group of order $2(p-1)$. This result holds over any algebraically closed field; see \cite{VM}.

The interesting question whether the genus $(p-1)^2$ together with an automorphism group as in (\ref{Konto}) characterize the Artin-Mumford curve has been solved so far only for curves defined over $\mathbb{F}_p$; see \cite{AK}.

A natural generalization of Artin-Mumford curves arises when  the polynomials $X^p-X$ and $Y^p-Y$ in (\ref{AM}) are replaced by separable linearized polynomials $L_1,L_2$ of equal degree. Our aim is to investigate such generalized Artin-Mumford curves, especially their automorphism groups. To present our results, we need some notation that will also be used throughout the paper.

For an odd prime $p$ and powers $\bar{q}=p^n$ and $q=\bar{q}^m$, $\mathbb F_p$, $\mathbb F_{\bq}$, $\mathbb F_q$ are the finite fields of order $p$, $\bq$, $q$; $\mathbb K$ is the algebraic closure of $\mathbb F_p$;
$L_1(T), L_2(T) \in \K[T]$ are separable polynomials of degree $q$ which are $\bq$-linearized. We admit that one, but not both, is $\bq^k$-linearized, for some $k \geq 2$.
With this notation, the \textit{generalized Artin-Mumford curve} $\cX_{(L_1, L_2)}$ is the nonsingular model of the plane curve with affine equation
\begin{equation}
\label{l1l2}
\cX_{(L_1, L_2)}:\quad  L_1(X) \cdot L_2(Y)=1.
\end{equation}
The family of generalized Artin-Mumford curves is denoted by:
$$\mathcal{S}_{q|\bq}=\big\{\cX_{(L_1,L_2)} \mid  L_1(T),L_2(T) \in \K[T], \ \deg(L_1)=\deg(L_2)=q, \ L_1,L_2 { \rm \  are \ separable,}$$ $$ \ \bq{\rm -linearized, \ not \ both} \  \bq^k {\rm -linearized \ for \ any \ }k\geq2 \big\}.$$

An interesting feature of a generalized Artin-Mumford curve $\cX_{(L_1, L_2)}$ is that its genus only depends on $q$, namely $\gg(\cX_{(L_1, L_2)})=(q-1)^2$. Also, $\cX_{(L_1, L_2)}$ is an ordinary curve, that is, its genus and $p$-rank are equal.
A complete description of the automorphism group of any generalized Artin-Mumford curve is given in the following two theorems.
\begin{theorem} \label{th1}
The full automorphism group of $\cX_{(L,L)}$ is the semidirect product
\begin{equation}
\label{aut1}
\Sigma \rtimes \Delta,
\end{equation}
where
\begin{itemize}
\item $\Sigma= \{\tau_{\alpha,\beta}:(X,Y)\mapsto(X+\alpha,Y+\beta) \mid L(\alpha)=L(\beta)=0 \}$ is an elementary abelian $p$-group of order $q^2$;
\item $\Delta=\langle \theta,  \xi \rangle$ is a dihedral group of order $2(\bq -1)$, where $\theta: (X,Y) \mapsto (\lambda X, \lambda^{-1}Y)$ with $\lambda$ a primitive $(\bq-1)$-th root of unity, and $\xi : (X,Y) \mapsto (Y,X)$.
\end{itemize}
\end{theorem}

\begin{theorem} \label{th2}
If $L_1 \ne L_2$, the full automorphism group of $\cX_{(L_1, L_2)}$ is the semidirect product
\begin{equation}
\label{aut2}
\Sigma \rtimes \Gamma,
\end{equation}
where
\begin{itemize}
\item $\Sigma= \{\tau_{\alpha,\beta}:(X,Y)\mapsto(X+\alpha,Y+\beta) \mid L_1(\alpha)=L_2(\beta)=0 \}$ is an elementary abelian $p$-group of order $q^2$;
\item $\Gamma=\langle \theta \rangle$ is a cyclic group of order $\bq -1$, where $\theta: (X,Y) \mapsto (\lambda X, \lambda^{-1}Y)$ with $\lambda$ a primitive $(\bq-1)$-th root of unity.
\end{itemize}
\end{theorem}
For $\bq=q$, the size of $\aut(\cX_{(L_1,L_2)})$ is approximately $2(\gg(\cX_{(L_1,L_2)})+1)^{3/2}$. Since the groups given in Theorems \ref{th1} and \ref{th2} are solvable, $\cX_{(L_1,L_2)}$ attains, up to the constant, the bound given in \cite{KM}.

Our main result is that $\aut(\cX_{(L_1,L_2)})$ together with $\gg(\cX_{(L_1,L_2)})$ characterize the curves in $\mathcal{S}_{q \mid \bq}$. This result can be viewed as a generalization of \cite[Theorem 1.1]{AK} on Artin-Mumford curves.

\begin{theorem} \label{th3}
Let $\cX$ be a $($projective, non-singular, geometrically irreducible, algebraic$)$ curve of genus $\gg=(q-1)^2$ defined over $\K$. If $\aut(\cX)$ contains an elementary abelian subgroup $E_{q^2}$ of order ${q^2}$, then $\cX$ is birationally equivalent over $\K$ to some $\cX_{(L_1,L_2)} \in \mathcal{S}_{q \mid \bq}$, where $\bq$ is the largest power of $p$ such that $\aut(\cX)$ contains a cyclic subgroup $C_{\bq-1}$ of order $\bq-1$.
\end{theorem}

In the case $L_1=L_2$, the assumption on the genus can be weakened under a stronger assumption on the automorphism group, as follows.

\begin{theorem} \label{th4}
Let $\cX$ be a curve of genus $\gg\leq(q-1)^2$ defined over $\K$. If $\aut(\cX)$ contains a semidirect product $E_{q^2}\times(C_2\times C_2)$ (where $E_{q^2}$ is elementary abelian of order $q^2$ and $C_2\times C_2$ is a Klein four-group), then $\cX$ is birationally equivalent over $\K$ to some $\cX_{(L,L)} \in \mathcal{S}_{q \mid \bq}$, where $\bq$ is the largest power of $p$ such that $\aut(\cX)$ contains a cyclic subgroup $C_{\bq-1}$ of order $\bq-1$.
\end{theorem}

In Section \ref{sec2}, preliminary results on automorphism groups of ordinary curves and curves of even genus are collected. In Section \ref{sec3}, we give the proofs of Theorems \ref{th1} and \ref{th2}, doing so we also show the relevant properties of generalized Artin-Mumford curves; see Lemma \ref{lem1}. The proof of Theorems \ref{th3} and \ref{th4} is given in Section \ref{sec4} where additional classification results of independent interest are found, as well. Here we only mention that Theorem \ref{vanvan} gives the following characterization.

\begin{theorem} \label{th5}
Let $\cY$ be a curve of genus $q-1$ defined over $\K$ whose automorphism group $\aut(\cY)$ contains an elementary abelian subgroup $E_q$ of order $q$.
Then one of the following holds.
\begin{itemize}
\item[(I)] $\cY$ is birationally equivalent over $\K$ to the curve $\cY_{L,a}$ with affine equation
$$
 L(y)=ax+\frac{1}{x},
$$
for some $a \in \K^*$ and $L(T) \in \K[T]$ a separable $p$-linearized polynomial of degree $q$. For the curve $\cY_{L,a}$ the following properties hold:
\begin{itemize}
\item[(i)] $\cY_{L,a}$ is ordinary and hyperelliptic;
\item[(ii)] $\cY_{L,a}$ has exactly $2q$ Weierstrass places, which are the fixed places of the hyperelliptic involution $\mu$.
\item[(iii)] The full automorphism group $\aut(\cY_L)$ of $\cY_{L,a}$ has order $4q$ and is a direct product $Dih(E_q)\times\langle\mu\rangle$.
\end{itemize}
\item[(II)]  $p\ne3$ and $\cY$ is birationally equivalent over $\K$ to the curve $\cZ_{\tilde L,b}$ with affine equation
$$
\tilde L(y)=x^3+bx,
$$
for some $a\in\K$ and $\tilde L(T)\in\K[T]$ a separable $p$-linearized polynomial of degree $q$.
For the curve $\cZ_{\tilde L,b}$ the following properties hold:
\begin{itemize}
\item[(i)] $\cZ_{\tilde L,b}$ has zero $p$-rank;
\item[(ii)] $\aut(\cZ_{\tilde L,b})$ contains a generalized dihedral subgroup $Dih(E_q)=E_q\rtimes \langle \nu \rangle$.
\end{itemize}
\end{itemize}
\end{theorem}

Theorem \ref{th5} provides a generalization of \cite[Proposition (2.2) and Corollary (2.3)]{vandervander}. 

Our proof uses function field theory, especially the Hurwitz genus formula and the Deuring-Shafarevich formula, together with deeper results on finite groups, especially the classification theorem on finite non-abelian simple groups whose Sylow $2$-subgroups are dihedral or semidihedral. In doing so we adopt the approach worked out by Giulietti and Korchm\'aros in \cite{GK2016}.

\section{Background and Preliminary Results}\label{sec2}

We keep the notation used in Introduction. Also, $\cX$ is a (projective, non-singular,
geometrically irreducible, algebraic) curve of genus $\gg\geq 2$ defined over $\mathbb{K}$, $\K(\cX)$ is the function field of $\cX$, and $\aut(\cX)$ is its full automorphism group over $\K$.

For a subgroup $G$ of $\aut(\cX)$, let $\bar \cX$ denote a non-singular model of $\K(\cX)^G$, that is,
a curve with function field $\K(\cX)^G$, where $\K(\cX)^G$ consists of all elements of $\K(\cX)$ fixed by every element in $G$. Usually, $\bar \cX$ is called the
quotient curve of $\cX$ by $G$ and denoted by $\cX/G$. The field extension $\K(\cX)|\K(\cX)^G$ is Galois of degree $|G|$.

Let $\Phi$ be the cover of $\cX|\bar{\cX}$ where $\bar{\cX}=\cX/G$. A place $P$ of $\K(\cX)$ is a ramification place of $G$ if the stabilizer $G_P$ of $P$ in $G$ is nontrivial; the ramification index $e_P$ is $|G_P|$. The $G$-orbit of $P$ in $\K(\cX)$ is the subset
$o=\{R\mid R=g(P),\, g\in G\}$ of the set of the places of $\K(\cX)$, and it is {\em long} if $|o|=|G|$, otherwise $o$ is {\em short}. For a place $\bar{Q}$, the $G$-orbit $o$ lying over $\bar{Q}$ consists of all places $P$ of $\K(\cX)$ such that $\Phi(P)=\bar{Q}$. If $P\in o$ then $|o|=|G|/|G_P|$ and hence $P$ is a ramification place if and only if $o$ is a short $G$-orbit. If every non-trivial element in $G$ is fixed--point-free on the set of the places of $\K(\cX)$, the cover $\Phi$ is unramified. For a non-negative integer $i$, the $i$-th ramification group of $\cX$
at $P$ is denoted by $G_P^{(i)}$ and defined to be
$$G_P^{(i)}=\{\alpha\in G_P\mid v_P(\alpha(t)-t)\geq i+1\}, $$ where $t$ is a local parameter at
$P$; see \cite{Sti}. Here $G_P^{(0)}=G_P$.

Let $\bar \gg$ be the genus of the quotient curve $\bar{\cX}=\cX/G$. The Hurwitz
genus formula \cite[Theorem 7.27]{HKT}  gives the following equation
    \begin{equation}
    \label{eq1}
2\gg-2=|G|(2\bar{\gg}-2)+\sum_{P\in \cX} d_P,
    \end{equation}
    where the different $d_P$ at $P$ is given by
\begin{equation}
\label{eq1bis}
d_P= \sum_{i\geq 0}(|G_P^{(i)}|-1),
\end{equation}
see \cite[Theorem 11.70]{HKT}. Let $\gamma$ and $\bar\gamma$ be the $p$-ranks of $\cX$ and $\bar\cX$ respectively.
The Deuring-Shafarevich formula \cite[Theorem 11.62]{HKT} states that
\begin{equation}
    \label{eq2deuring}
\gamma-1={|G|}(\bar{\gamma}-1)+\sum_{i=1}^k (|G|-\ell_i)
    \end{equation}
where $\ell_1,\ldots,\ell_k$ are the sizes of the short orbits of $G$.

A subgroup $G$ of $\aut(\cX)$ is {\em{tame}} if $\gcd(p,|G|)=1$, otherwise $G$ is {\em{non-tame}}.
The stabilizer $G_P$ of a place $P\in \cX$ in $G$ is a semidirect product $G_P=Q_P\rtimes U$ where the normal subgroup $Q_P$ is a $p$-group while the complement $U$ is a tame cyclic group; see \cite[Theorem 11.49]{HKT}.

The following result is due to Nakajima; see \cite[Theorems 1, 2 and 3]{N} and \cite[Lemma 11.75]{HKT}.

\begin{theorem} \label{naka} Let $\cX$ be a curve with $\gg(\cX) \geq 2$ defined over an algebraically closed field of characteristic $p \geq 3$, and $H$ be a Sylow $p$-subgroup of $\aut(\cX)$. Then the following hold.
\begin{itemize}
\item[(I)] When $\gamma(\cX) \geq 2$, we have
$$|H| \leq \frac{p}{p-2} (\gamma(\cX)-1) \leq \frac{p}{p-2} (\gg(\cX)-1).$$
\item[(II)] If $\cX$ is ordinary $($i.e. $\gg(\cX)=\gamma(\cX))$ and $G \leq \aut(\cX)$, then $G_P^{(2)}=\{1\}$ and $G_P^{(1)}$ is elementary abelian, for every $P \in \cX$.
\item[(III)] If $\cX$ is ordinary then $|\aut(\cX)| \leq 84(\gg(\cX)-1)\gg(\cX)$.
\item[(IV)] If $\gamma(\cX)=1$ then $H$ is cyclic.
\end{itemize}
\end{theorem}

The following results are due to Giulietti and Korchm\'aros; see \cite{GK2016}.

\begin{lemma}
\label{terribile} Let $H$ be a solvable automorphism group of an algebraic curve $\cX$ of genus $\gg(\cX) \ge 2$ containing a normal $d$-subgroup $Q$ of odd order such that $|Q|$ and $[H:Q]$ are coprime. Suppose that a complement $U$ of $Q$ in $H$ is abelian, and that $N_H(U)\cap Q=\{1\}$.  If
\begin{equation}
\label{eq22bisdic2015}
{\mbox{$|H|\geq 30(\gg(\cX)-1)$}},
\end{equation}
then  $d=p$ and $U$ is cyclic.
\end{lemma}

The {\emph{odd core}} $O(G)$ of a group $G$ is its maximal normal subgroup of odd order. If $O(G)$ is trivial, then $G$ is an {\emph{odd core-free}} group.

\begin{lemma}
\label{oddcore} Let $\cX$ be a curve of even genus, and $G$ be an odd core-free automorphism group of $\cX$ with a non-abelian simple minimal normal subgroup $M$. Up to isomorphism, one of the following cases occurs for some prime $d$ and odd $k$:
\begin{itemize}
\item[\rm(i)] $M=PSL(2,d^k)\le G \le P\Gamma L(2,d^k)$ with $d^k \ge 5$;
\item[\rm(ii)] $M=PSL(3,d^k)\le G \le P\Gamma L(3,d^k)$ with $d^k \equiv 3 \pmod 4$;
\item[\rm(iii)] $M=PSU(3,d^k)\le G \le P\Gamma U(3,d^k)$ with $d^k \equiv 1 \pmod 4$;
\item[\rm(iv)] $M=G=A_7$, the alternating group on $7$ letters;
\item[\rm(v)]  $M=G=M_{11}$, the Mathieu group on $11$ letters.
\end{itemize}
\end{lemma}

\begin{lemma}
\label{elem8} If $\cX$ is a curve of even genus then $\aut(\cX)$ has no elementary abelian $2$-subgroup of order $8$.
\end{lemma}

\begin{lemma}
\label{norm2} Let $\cX$ be a curve of even genus and $G \leq \aut(\cX)$. If $G$ has a minimal normal subgroup of order $2$ then $G=O(G)\rtimes S_2$, where $S_2$ is Sylow $2$-subgroup of $G$, unless $S_2$ is a generalized quaternion group.
\end{lemma}

For a positive integer $d$, $C_d$ stands for a cyclic group of order $d$, $D_d$ for a dihedral group of order $2d$, $E_d$ for an elementary abelian group of order $d$, and $Dih(E_d)$ for a generalized dihedral group $E_d \rtimes C_2$ of order $2d$.

\section{The automorphism group of $ \cX_{(L_1, L_2)}$} \label{sec3}

\begin{lemma} \label{lem1}
For the curve $\cX_{(L_1, L_2)}$ as in {\rm (\ref{l1l2})}, $X_\infty = (1:0:0)$ and $Y_\infty=(0:1:0)$, the following properties hold:
\begin{itemize}
\item[i)] $X_\infty$ and $Y_\infty$ are $q$-fold ordinary points;
\item[ii)] $\cX_{(L_1, L_2)}$ is ordinary with $\gg(\cX_{(L_1, L_2)})=\gamma(\cX_{(L_1, L_2)})=(q-1)^2$;
\item[iii)] If $L_1 \ne L_2$, $\aut(\cX_{(L_1, L_2)})$ contains the subgroup $\Sigma \rtimes \Gamma$ defined in {\rm (\ref{aut2})};
\item[iv)] If $L_1=L_2=L$, $\aut(\cX_{(L, L)})$ contains the subgroup $\Sigma \rtimes \Delta$ defined in {\rm (\ref{aut1})};
\item[v)] In both cases {\rm iii)} and {\rm iv)}, the group $\Sigma$ is a Sylow $p$-subgroup of $\aut(\cX_{(L_1, L_2)})$.
\item[vi)] The quotient curves $\cX_{(L_1, L_2)}/\Sigma_x$ and $\cX_{(L_1, L_2)}/\Sigma_y$ are rational curves, where $\Sigma_x=\{\tau_{\alpha, \beta}\in\Sigma\mid \beta=0\}$ and $\Sigma_y=\{\tau_{\alpha, \beta}\in\Sigma\mid \alpha=0\}$.
\end{itemize}
\end{lemma}

\begin{proof}
Let $\bar{P}_{x=\alpha_i}$, with $L_1(\alpha_i)=0$, be the $q$ distinct zeros and $\bar{P}_{x=\infty}$ be the unique pole of $L(x)$ in $\K(x)$. Then
$$ v_{\bar{P}_{x=\alpha_i}}(1/L_1(x))=-1,\quad v_{\bar{P}_{x=\infty}}(1/L_1(x))=q, $$
and $1/L_1(x)$ has valuation zero at any other place of $\K(x)$. Thus, the function field $\K(\cX_{(L_1, L_2)})=\K(x,y)$ with $L_1(x) \cdot L_2(y)=1$, is a generalized Artin-Schreier extension of $\K(x)$ of degree $q$; see \cite[Proposition 3.7.10]{Sti}. The places $\bar{P}_{x=\alpha_i}$ are totally ramified while any other place is unramified. The genus of $\cX_{(L_1, L_2)}$ is given by
$$\gg(\cX_{(L_1, L_2)}) = q \cdot \gg(\K(x))+ \frac{q-1}{2} \cdot (-2+2q)=(q-1)^2.$$
The places $P_{x=\alpha_i}$ lying over $\bar{P}_{x=\alpha_i}$, $i=1 , \ldots , q$, are the poles of $y$ and they are centered at $Y_\infty$. The unique zero of $y$ is place $P_{x=\infty}$ lying over $\bar{P}_{x=\infty}$. Analogously, $x$ has $q$ distinct poles $P_{y=\beta_i}$, with $L_2(\beta_i)=0$, which are simple and centered at $X_\infty$, and a unique zero $P_{y=\infty}$.
Note that $P_{x=\infty}=P_{y=0}$ and $P_{y=\infty}=P_{x=0}$. Let $\Sigma=\{\tau_{\alpha,\beta}:(X,Y)\mapsto(X+\alpha,Y+\beta) \mid L_1(\alpha)=L_2(\beta)=0 \}$. By direct computation $\Sigma$ is an elementary abelian $p$-subgroup of $\aut(\cX_{(L_1, L_2)})$ of order $q^2$. From Theorem \ref{naka}(I), $\Sigma$ is a Sylow $p$-subgroup of $\aut(\cX_{(L_1, L_2)})$. Thus the Galois group of $\K(x,y)|\K(x)$ is contained in $\Sigma$ up to conjugation, and hence $\K(x,y)^\Sigma$ is rational. By direct computation $\Sigma$ has at least two short orbits of length $q$, namely
$$\Omega_x=\{P_{y=\beta} \mid L_2(\beta)=0 \}, \quad \Omega_y=\{P_{x=\alpha} \mid L_1(\alpha)=0\}.$$
From the Deuring-Shafarevich formula (\ref{eq2deuring}) applied to the extension $\K(x,y)|\K(x,y)^\Sigma$,
$$q^2-2q=\gg(\cX_{(L_1, L_2)})-1 \geq \gamma(\cX_{(L_1, L_2)})-1 \geq q^2(0-1)+2(q^2-q)=q^2-2q.$$
Therefore the curve $\cX_{(L_1, L_2)}$ is ordinary. By direct checking, if $L_1 \ne L_2$, then $\Sigma$ and $\Gamma$ are subgroups of $\aut(\cX_{(L_1, L_2)})$, $\Gamma$ normalizes $\Sigma$, and $\Gamma \cap \Sigma=\{1\}$. Analogously, if $L_1 = L_2$, then $\Sigma$ and $\Delta$ are subgroups of $\aut(\cX_{(L_1, L_2)})$, $\Delta$ normalizes $\Sigma$, and $\Delta \cap \Sigma=\{1\}$.

In order to prove vi), set $\eta=L_1(x)$. Then $\K(\eta,y)\subseteq\K(\cX_{(L_1 ,L_2)})^{\Sigma_x}$. Since $[\K(\cX_{(L_1, L_2)}):\K(\eta,y)]\leq q$, this implies $\K(\cX_{(L_1, L_2)})^{\Sigma_x}=\K(\eta,y)$ and
$$ \cX_{(L_1, L_2)}/\Sigma_x:\; L_2(y)=\frac{1}{\eta}. $$
This shows that $\cX_{(L_1, L_2)}/\Sigma_x$ is rational, and the same holds for $\cX_{(L_1, L_2)}/\Sigma_y$.
\end{proof}

The following result follows from the proof of Lemma \ref{lem1}.

\begin{corollary} \label{cor2}
The group $\Sigma$ has exactly two short orbit $\Omega_x$ and $\Omega_y$, both of length $q$. Namely, $$\Omega_x=\{P_{y=\beta} \mid L_2(\beta)=0 \}, \quad \Omega_y=\{P_{x=\alpha} \mid L_1(\alpha)=0\}.$$
Moreover $\K(x,y)^\Sigma$ is rational and the principal divisors of the coordinate functions are given by
$$(x)=q\,P_{y=0} -\sum_{P \in \Omega_y} P, \quad (y)=q\,P_{x=0} -\sum_{P \in \Omega_x} P.$$
\end{corollary}

\begin{lemma} \label{lem3}
Let $C$ be a cyclic subgroup of $\aut(\cX_{(L_1,L_2)})$ containing $\Gamma= \langle \theta \rangle$, where  $\theta: (X,Y) \mapsto (\lambda X, \lambda^{-1}Y)$ with $\lambda$ a primitive $(\bq-1)$-th root of unity. Suppose that $C$ is contained in the normalizer $N$ of $\Sigma$ in $\aut(\cX_{(L_1,L_2)})$. Then $C=\Gamma$.
\end{lemma}

\begin{proof}
First of all we observe that $C \cap \Sigma=\{1\}$. In fact by direct checking $\Gamma$ does not commute with any non trivial $p$-element $\tau_{\alpha,\beta} \in \Sigma$. From Lemma \ref{lem1} v), $C$ is tame. Since $C \leq N$, $C$ is isomorphic to an automorphism group $\bar{C}$ of $\cX_{(L_1,L_2)} / \Sigma$. Denote by $\bar{\Gamma}$ the subgroup of $PGL(2,\K)$  which is isomorphic to $\Gamma$. Moreover, from Corollary \ref{cor2}, $C$ acts on $\Omega_x \cup \Omega_y$, and $\bar{C} \leq PGL(2,\K)$ as $\cX_{(L_1,L_2)} / \Sigma$ is rational. From \cite[Hauptsatz 8.27]{Hup} both $\bar{C}$ and $\bar{\Gamma}$ fix exactly two places on $\cX_{(L_1,L_2)} / \Sigma$ which are then the two places $\bar{P}_x$ and $\bar{P}_y$ lying under $\Omega_x$ and $\Omega_y$ respectively. Hence, from Corollary \ref{cor2}, $C$ fixes the pole divisors of $x$ and $y$. From the Orbit stabilizer theorem $C$ fixes at least one place in $\Omega_x$ and one place in $\Omega_y$. By direct computation $\Gamma$ fixes $P_{x=0} \in \Omega_y$ and $P_{y=0} \in \Omega_x$, acting semiregularly on $\Omega_x \setminus \{P_{y=0}\}$ and $\Omega_y \setminus \{P_{x=0}\}$. Thus, $C$ fixes $P_{y=0}$ and $P_{x=0}$ and hence the zero divisors of $x$ and $y$ are preserved by $C$ from Corollary \ref{cor2}. This implies that the generator $c$ of $C$ has the form $c: \  (x,y) \mapsto (\gamma x, \delta y),$
for some $\gamma, \delta \in \K$. By direct computation $\gamma^{\bq-1}=\delta^{\bq-1}=1$, and so $C=\Gamma$.
\end{proof}

\begin{corollary} \label{cordamette}
Let $C$ be a cyclic subgroup of the normalizer $N$ of $\Sigma$ in $\aut(\cX_{(L_1, L_2)})$ such that $(\bq-1)\mid |C|$ and $|C| \mid (q-1)$. Then $C=\Gamma$.
\end{corollary}

\vspace*{1 mm}
\subsection{Proof of Theorem \ref{th1}}
\ \\ \\
In this section, $L_1=L_2=L$ and we refer to $\Sigma$ and $\Delta$ as defined in Theorem \ref{th1}. For $q=p$ Theorem \ref{th1} was proved in \cite[Theorem 1.1]{AK}. Thus, we suppose that $q>p$.

\begin{lemma} \label{lem4}
The normalizer $N$ of $\Sigma$ in $\aut(\cX_{(L,L)})$ is $N=\Sigma \rtimes \Delta$.
\end{lemma}
\begin{proof}
From Corollary \ref{cor2}, $\bar{N}=N/\Sigma$ is a tame subgroup of $PGL(2,\K)$ containing a dihedral group $\bar{\Delta}$ which is isomorphic to $\Delta=\Gamma \rtimes \langle \xi \rangle$, where $\Gamma=\langle \theta \rangle$. Now we show that there are no involutions in $N \setminus (\Sigma \rtimes \Delta)$. Let $\iota \in N$ be an involution and let $\bar{\iota}$ be the induced involution in $PGL(2,\K)$. Denote by $\bar{P}_x$ and $\bar{P}_y$ the places lying under $\Omega_x$ and $\Omega_y$ respectively. From \cite[Hauptsatz 8.27]{Hup} there exists a unique involution in $PGL(2,\K)$ fixing $\bar{P}_x$ and $\bar{P}_y$, and it is induced by $\theta^{(\bq-1)/2}$. Thus, if $\iota \not\in \Gamma$ then $\iota$ switches $\Omega_x$ and $\Omega_y$. From Corollary \ref{cor2}, $\iota$ maps $x$ to $a(y+\alpha)$ and $y$ to $b(x+\beta)$ where $a,b \in \K$ and $L(\alpha)=L(\beta)=0$. Since the order of $\iota$ is equal to $2$, we have that $\alpha=\beta=0$ and $\alpha=\beta \in \{-1,1\}$. Hence, $\iota=\xi$ or $\iota=\theta^{(\bq-1)/2} \cdot \xi$, and so $\iota \in \Delta$. From \cite[Hauptsatz 8.27]{Hup}, one of the following holds:
\begin{enumerate}
\item $\bar{N}$ is isomorphic either to $A_4$ or $S_4$ or $A_5$. 
\item $\bar{N}$ is isomorphic to a dihedral group $D_d$ of order $2d$.
\end{enumerate}
Suppose $\bar{N} \cong A_4$. If $\bq \ne 3$, $\bar{\Delta}$ is not contained in $\bar{N}$. If $\bq=3$ then $\bar{N}$ is not tame, a contradiction.

Suppose $\bar{N} \cong S_4$. In this case $\bq=3$, which is impossible as $\bar{N}$ is tame, or $\bq=5$, which is impossible as $\bar{N}$ contains more than the $5$ involutions contained in $\bar{\Delta} \cong D_{8}$.

Suppose that $\bar{N} \cong A_5$. Then as before $\bq=3$ which is not possible.


Therefore, case (2) occurs. From Lemma \ref{lem3}, $d=\bq-1$ and the claim follows.
\end{proof}

In order to prove that $\aut(\cX_{(L,L)})=N$, several cases are distinguished according to the structure of the minimal normal subgroups of $\aut(\cX_{(L,L)})$. Recall that every finite group admits a minimal normal subgroup, which is either elementary abelian or a direct product of isomorphic simple groups.

\begin{lemma} \label{lem5}
If $\aut(\cX_{(L,L)})$ has a minimal normal subgroup $E_{d^k}$ which is an elementary abelian $d$-group, then $\aut(\cX_{(L,L)})$ admits an elementary abelian minimal normal subgroup $M$ which is a $p$-group.
\end{lemma}

\begin{proof}
Assume that $d \ne p$. Since $\Sigma$ normalizes $E_{d^k}$ and $\gcd(d,p)=1$, we have $H=\langle \Sigma, E_{d^k} \rangle = E_{d^k} \rtimes \Sigma$. From Lemma \ref{terribile}, either  $|E_{d^k} \rtimes \Sigma|< 30(\gg(\cX_{(L,L)})-1)$ or $N_{H}(\Sigma) \cap E_{d^k}= E_{d^h} \ne \{1\}$ with $0<h \leq k$.
\begin{itemize}
\item Assume that $N_{H}(\Sigma) \cap E_{d^k}= E_{d^h} \ne \{1\}$ with $0<h \leq k$. From Lemma \ref{lem4}, $E_{d^h} \leq \Delta$ up to conjugation and hence $d^h=4$ or $h=1$. If $d^h=4$, then $E_{d^h}=E_{d^k}=\langle \xi \rangle \times \langle \theta^{\frac{\bq-1}{2}} \rangle$ from Lemma \ref{elem8}. By direct checking $E_{d^k}$ does not commute with $\Sigma$, a contradiction. Hence $E_{d^h}=C_d \leq C_{\bq-1}$.
If $d=2$ then $\aut(\cX_{(L,L)})=O(\aut(\cX_{(L,L)})) \rtimes S_2$ by Lemma \ref{norm2}. Thus $O(\aut(\cX_{(L,L)}))$ contains a minimal normal subgroup of $\aut(\cX_{(L,L)})$, and we can assume $d$ to be odd. Assume that $d \ne p$ is odd. Since $C_d \leq \Gamma$ and $E_{d^k}$ is abelian, we have that $E_{d^k}$ fixes $P_{y=0}$ and $P_{x=0}$, and acts on $\Omega_x \setminus \{P_{y=0}\}$ and $\Omega_y \setminus \{P_{x=0}\}$. Arguing as in the proof of Lemma \ref{lem3}, $E_{d^k} \leq \Gamma$.
Hence $E_{d^k}=C_d$ which cannot commute with $\Sigma$, a contradiction.
\item Assume that $|E_{d^k} \rtimes \Sigma|< 30(\gg(\cX_{(L,L)})-1)$. By direct computation $d^k<30$. Since no subgroup of $\Sigma$ commutes with $E_{d^k}$ we have that $\Sigma$ is isomorphic to a subgroup of $GL(k,d)$. If $d^k \ne 27$ then $GL(k,d)$ has no elementary abelian subgroup of odd square order. If $d^k=27$ then $d=p=3$, a contradiction.
\end{itemize}
\end{proof}

\begin{remark} We have shown in Lemma \ref{lem5} that $\aut(\cX_{(L,L)})$ does not admit elementary abelian normal $d$-subgroups for $d \ne p$ odd. If $\aut(\cX_{(L,L)})$ admits an elementary abelian normal $2$-subgroup then it also admits a minimal normal $p$-subgroup.
\end{remark}

\begin{proposition} \label{pro7}
If $\aut(\cX_{(L,L)})$ admits an elementary abelian minimal normal subgroup $M$, then $\aut(\cX_{(L,L)})=\Sigma \rtimes \Delta$.
\end{proposition}

\begin{proof}
From Lemma \ref{lem5}, we can assume that $M\leq\Sigma$.
Let $\tilde\Sigma$ be a Sylow $p$-subgroup of $\aut(\cX_{(L,L)})$.
Then $M\subseteq\Sigma\cap\tilde\Sigma$. For any $\tau_{\alpha\beta}\in M$ and $\sigma\in\aut(\cX_{(L,L)})$, we have $\sigma(\tau_{\alpha\beta})=\tau_{\alpha^\prime\beta^\prime}$ for some $\alpha^\prime,\beta^\prime$. Therefore $\sigma$ acts on the poles of $x$ and on the poles of $y$, that is, $\sigma$ acts on $\Omega_y$ and on $\Omega_x$.
Suppose by contradiction that there exists $\omega$ in $\Sigma\setminus\tilde\Sigma$ fixing a place $P\in\Omega_x\cup\Omega_y$. Then $\aut(\cX_{(L,L)})$ admits a Sylow $p$-subgroup $\bar\Sigma$ containing $\omega$ and the stabilizer $\tilde\Sigma_P$ of $P$ in $\tilde\Sigma$. Thus the order of $\bar\Sigma_P$ is strictly greater than the order of $\tilde\Sigma_P$, a contradiction. This proves that $\Sigma_P=\tilde\Sigma_P$ for all $P\in\Omega_x\cup\Omega_y$, and hence $\Sigma=\tilde\Sigma$. The claim follows from Lemma \ref{lem4}.
\end{proof}

\begin{proposition} \label{prop8}
$\aut(\cX_{(L,L)})$ admits an elementary abelian minimal normal subgroup.
\end{proposition}

\begin{proof}
Suppose by contradiction that $\aut(\cX_{(L,L)})$ admits no elementary abelian minimal normal subgroup. Thus, $\aut(\cX_{(L,L)})$ is odd-core free. In fact if $O(\aut(\cX_{(L,L)})) \ne \{1\}$ then $O(\aut(\cX_{(L,L)}))$ contains a minimal normal subgroup which is then elementary abelian by the Feit-Thompson theorem.
From Lemma \ref{oddcore} one of the following cases occurs:
\begin{itemize}
\item[\rm(i)] $M:=PSL(2,d^k)\trianglelefteq \aut(\cX_{(L,L)})  \leq P\Gamma L(2,d^k)$. In this case $\Sigma / (\Sigma \cap M)$ is isomorphic to a subgroup of $P\Gamma L(2,d^k) / PSL(2,d^k)$. Since $[PGL(2,d^k):PSL(2,d^k)]=2$ and $P\Gamma L(2,d^k) / PGL(2,d^k)$ is cyclic of order $k$, we have that  $\Sigma / (\Sigma \cap M)$ is cyclic. Then either  $\Sigma / (\Sigma \cap M)=\{1\}$ or  $\Sigma / (\Sigma \cap M)=C_p$. When $r$ is an odd prime, the Sylow $r$-subgroups of $PSL(2,d^k)$ are cyclic unless $r=d$. Since $q>p$, this implies that $d=p$ and either $d^k=q^2$ or $d^k=q^2/p$.
In both cases, arguing as in the proof of Proposition \ref{pro7}, we have that any element of $\aut(\cX_{(L,L)})$ normalizing $\Sigma \cap M$ normalizes the whole group $\Sigma$. Therefore from \cite[Hauptsatz 8.27]{Hup} $\aut(\cX_{(L,L)})$ contains a cyclic group of order $q^2-1$ or $q^2/p -1$ normalizing $\Sigma$, a contradiction to Lemma \ref{lem4}.
\item[\rm(ii)] $M:=PSL(3,d^k)\trianglelefteq \aut(\cX_{(L,L)}) \leq P\Gamma L(3,d^k)$. We have $[PGL(3,d^k):PSL(3,d^k)] \in \{1,3\}$ and $P\Gamma L(3,d^k) / PGL(3,d^k)$ is cyclic of order $k$. Hence $\Sigma / (\Sigma \cap M)$ is cyclic. Then either  $\Sigma / (\Sigma \cap M)=\{1\}$ or  $\Sigma / (\Sigma \cap M)=C_p$. If $d=p$ then a contradiction is obtained since a Sylow $d$-subgroup of $PSL(3,d^k)$ is not abelian. If either $\gcd(3,d^k-1)=1$, or $\gcd(3,d^k-1)=3$ and $p \ne 3$, then a contradiction follows from Lemma \ref{naka}. Suppose that $\gcd(3,d^k-1)=3$ and $p=3$. In this case a contradiction is obtained because the Sylow $3$-subgroup of $M$ is not abelian (see \cite[Satz 7.2]{Hup}), and hence cannot be contained in $\Sigma$.
\item[\rm(iii)] $M:=PSU(3,d^k)\trianglelefteq \aut(\cX_{(L,L)}) \leq P\Gamma U(3,d^k)$. We have $[PGL(3,d^k):PSL(3,d^k)] \in \{1,3\}$ and $P\Gamma L(3,d^k) / PGL(3,d^k)$ is cyclic of order $k$. Hence $\Sigma / (\Sigma \cap M)$ is cyclic. Then either  $\Sigma / (\Sigma \cap M)=\{1\}$ or  $\Sigma / (\Sigma \cap M)=C_p$. If $d=p$ then a contradiction is obtained since a Sylow $d$-subgroup of $PSL(3,d^k)$ is not abelian. If either $\gcd(3,d^k+1)=1$, or $\gcd(3,d^k+1)=3$ and $p \ne 3$, then a contradiction follows from Lemma \ref{naka}.
Suppose that $\gcd(3,d^k+1)=3$ and $p=3$. In this case a contradiction is obtained because the Sylow $3$-subgroup of $M$ is not abelian (see \cite[Theorem A.10 Case (iii)]{HKT}), and hence cannot be contained in $\Sigma$.
\item[\rm(iv)] $\aut(\cX_{(L,L)})=A_7$. Since $|A_7|=2^3\cdot3^2\cdot5\cdot7$, we have $q=3=p$, which is impossible.
\item[\rm(v)]  $\aut(\cX_{(L,L)})=M_{11}$. Since $|M_{11}|=2^4\cdot3^2\cdot5\cdot11$, we have $q=3=p$, which is impossible.
\end{itemize}
\end{proof}

From Propositions \ref{pro7} and \ref{prop8}, Theorem \ref{th1} follows.

\subsection{Proof of Theorem \ref{th2}}
\ \\ \\
In this section, $L_1 \ne L_2$ and we refer to $\Sigma$ and $\Gamma$ as defined in Theorem \ref{th2}.

\begin{lemma} \label{lem9}
The normalizer $N$ of $\Sigma$ in $\aut(\cX_{(L_1,L_2)})$ is $N=\Sigma \rtimes \Gamma$.
\end{lemma}
\begin{proof}
From Corollary \ref{cor2}, $\bar{N}=N/\Sigma$ is a tame subgroup of $PGL(2,\K)$ containing a cyclic group $\bar{\Gamma}$ which is isomorphic to $\Gamma$. Arguing as in the proof of Lemma \ref{lem4}, $N$ has no involution other than $\theta^{(\bq-1)/2}$, because by direct checking $\xi:(x,y) \mapsto (y,x)$ is not in $\aut(\cX_{(L_1,L_2)})$. From \cite[Hauptsatz 8.27]{Hup}, one of the following holds:
\begin{enumerate}
\item $\bar{N}$ is isomorphic either to $A_4$ or $S_4$ or $A_5$. 
\item $\bar{N}$ is isomorphic to a cyclic group $C_d$.
\end{enumerate}

Arguing as in the proof of Lemma \ref{lem4}, Case (1) is not possible because $\bar{N}$ is tame and it contains only one involution.
Therefore, case (2) occurs. From Lemma \ref{lem3}, $d=\bq-1$ and the claim follows.
\end{proof}

The proofs of the following results are analogous to the ones of Lemma \ref{lem5}, Proposition \ref{pro7}, and Proposition \ref{prop8}., and are omitted.

\begin{lemma} \label{lem10}
If $\aut(\cX_{(L_1,L_2)})$ has a minimal normal subgroup $E_{d^k}$ which is an elementary abelian $d$-group, then $\aut(\cX_{(L_1,L_2)})$ admits an elementary abelian minimal normal subgroup $M$ which is a $p$-group.
\end{lemma}

\begin{proposition} \label{pro11}
If $\aut(\cX_{(L_1,L_2)})$ admits an elementary abelian minimal normal subgroup, then $\aut(\cX_{(L_1,L_2)})=\Sigma \rtimes \Gamma$.
\end{proposition}

\begin{proposition} \label{prop12}
$\aut(\cX_{(L_1,L_2)})$ admits an elementary abelian minimal normal subgroup.
\end{proposition}

From Propositions \ref{pro11} and \ref{prop12}, Theorem \ref{th2} follows.

\section{Curves with automorphism group containing $E_{q^2}$} \label{sec4}

We need the following result on curves admitting $E_{q^2}$ as an automorphism group.

\begin{proposition} \label{prop44}
For a curve $\cX$ defined over $\K$, assume that one of the following holds.
\begin{itemize}
\item[(A)] $\cX$ has genus $\gg\leq(q-1)^2$ and the automorphism group $\aut(\cX)$ has a subgroup $H=E_{q^2} \rtimes (C_2\times C_2)$.
\item[(B)] $\cX$ has genus $\gg=(q-1)^2$ and the automorphism group $\aut(\cX)$ has a subgroup $H=E_{q^2}$.
\end{itemize}
Let $\{ M_i\}_i$ be the set of subgroups of $E_{q^2}$ of order $q$.
Then the following hold.
\begin{enumerate}
\item $\cX$ is an ordinary curve of genus $(q-1)^2$;
\item Up to relabeling the indeces, the cover $\cX \mid \cX / M_i$ is unramified for each $i \ne 1,2$;
\item $E_{q^2}$ has only two short orbits $\Omega_1$ and $\Omega_2$ on $\cX$, each of size $q$. The places of $\Omega_i$ share the same stabilizer $M_i$ for $i \in \{1,2\}$, and $M_1 \ne M_2$. Moreover, $\cX / M_1$ and $\cX / M_2$ are rational.
\end{enumerate}
\end{proposition}

\begin{proof}
Let $\gg$ and $\gamma$, $\bar\gg$ and $\bar\gamma$, be the genus and $p$-rank of $\cX$, $\bar\cX:=\cX/E_{q^2}$ respectively. Also, denote by $k\in\mathbb{N}$ the number of short orbits of $E_{q^2}$ on $\cX$, by $\Omega_i$ ($1\leq i \leq k$) the $i$-th short orbit of $E_{q^2}$, by $\ell_i\in\{p,p^2,\ldots,q^2/p\}$ the length of $\Omega_i$, and by $M_i$ the stabilizer of a given place $P_i\in\Omega_i$ in $E_{q^2}$, of size $q^2/\ell_i$.
Note that $M_i$ coincides with the stabilizer in $E_{q^2}$ of any place in $\Omega_i$, because $E_{q^2}$ acts on the fixed places of its normal subgroup $M_i$.
\begin{itemize}
\item[(A)] Case $\gg\leq(q-1)^2$ and $H:=E_{q^2} \rtimes (C_2\times C_2) \leq \aut(\cX)$.

If $\gamma=0$, then every element of $E_{q^2}$ fixes exactly one place of $\cX$ from \cite[Lemma 11.129]{HKT}. Since $E_{q^2}$ is abelian all elements of $E_{q^2}$ have the same fixed place $P$, which is fixed also by $H$. Thus, $H / E_{q^2}$ is cyclic by \cite[Theorem 11.49]{HKT}, a contradiction to $H / E_{q^2} \cong C_2\times C_2$. If $\gamma=1$ then $E_{q^2}$ is cyclic by Theorem \ref{naka} (IV), a contradiction. Hence $\gamma \geq 2$.
The Deuring-Shafarevich formula (\ref{eq2deuring}) applied to $E_{q^2}$ yields
\begin{equation}
\label{eqds}
\gamma-1 = q^2(\bar\gamma-1)+\sum_{i=1}^{k} (q^2-\ell_i).
\end{equation}
If $k=0$ then $\bar\gamma=(\gamma-1)/q^2 +1 > 1$, and hence $q^2\leq \gamma-1 \leq \gg-1 \leq q^2-2q$, a contradiction. Therefore $\bar\gamma \leq 1$ and $k \geq 1$.

Assume that $\bar\gamma=1$. The Riemann-Hurwitz formula together with $\bar\gg\geq\bar\gamma$ yields $\bar\gg=1$. If $k \geq 2$ then $\gamma-1 \geq 2(q^2-q^2/p)$ by equation (\ref{eqds}), a contradiction to $\gamma \leq \gg$. This yields $k=1$. Since $C_2\times C_2$ normalizes $E_{q^2}$ which has a unique short orbit $\Omega_1$, the induced group $\bar C_2\times\bar C_2$ fixes one place of the elliptic curve $\bar\cX$. From \cite[Theorem 11.94 (ii)]{HKT} and its proof, $\bar C_2\times\bar C_2$ is cyclic, a contradiction.

Therefore $\bar\gamma=0$. If $k \geq 3$ then equation (\ref{eqds}) together with $\gg \geq \gamma$ yields a contradiction. If $k=1$ then equation (\ref{eqds}) reads $2 \geq \gamma=1-\ell_1$, a contradiction. Thus $k=2$ and equation (\ref{eqds}) reads
$$\gamma=q^2+1-(\ell_1+\ell_2).$$
We prove that $\bar\gg=0$. From the Riemann-Hurwitz formula (\ref{eq1}) applied to $\cX\rightarrow\bar\cX$ we have that
$$q^2\bar\gg \leq \ell_1+\ell_2-2q \leq 2\frac{q^2}{p}-2q,$$
which implies $\bar\gg=0$.
Since $C_2\times C_2$ normalizes $E_{q^2}$, the induced group $\bar C_2\times\bar C_2$ is a subgroup of $PGL(2,\K)$ acting on the two places $\bar{P}_1$ and $\bar{P}_2$ lying under $\Omega_1$ and $\Omega_2$. From \cite[Hauptsatz 8.27]{Hup}, $\bar C_2\times\bar C_2$ switches $\bar{P}_1$ and $\bar{P}_2$ and hence $\ell_1=\ell_2=\ell$.
Let $P\in\Omega_i$. From \cite[Lemma 11.75 (v)]{HKT} either $(E_{q^2})_P^{(2)}$ is trivial, or $(E_{q^2})_P^{(2)}=E_{q^2}$, or $1<|(E_{q^2})_P^{(2)}|=\cdots=|(E_{q^2})_P^{(2)}|<q^2$.
By direct checking with the Riemann-Hurwitz formula applied to $\cX\rightarrow\bar\cX$, the second and the third case are not possible; hence $(E_{q^2})_P^{(2)}$ is trivial for all $P$, which implies $\ell=q$.
Now the Deuring-Shafarevich formula yields $\gamma=(q-1)^2\geq\gg$; hence, $\gamma=\gg=(q-1)^2$ and the claim {\it (1)} follows.
Since $M_i$, $i=1,2$, is the stabilizer in $E_{q^2}$ of any place in $\Omega_i$, we have that any other subgroup $M_j$ of order $q$ of $E_{q^2}$, $j\ne1,2$, has no fixed place, and thus the claim {\it (2)} is proved. Finally, for $i=1,2$, denote by $\gg_i$ the genus of the curve $\cX / M_i$. By the Riemann-Hurwitz formula (\ref{eq1}) applied to the cover $\cX \rightarrow \cX / M_i$,
\begin{equation}
\label{caz}
2\gg-2=2(q^2-2q)\geq 2q(\gg_i-1) +2q(q-1).
\end{equation}
Hence $\gg_i=0$ for $i=1,2$ and equality holds in (\ref{caz}). This proves that $M_i$ has no fixed place out of $\Omega_i$, and so $M_1 \ne M_2$.
\item[(B)]
Case $\gg=(q-1)^2$ and $H:=E_{q^2} \leq \aut(\cX)$.

Suppose $\gamma=0$. Then by \cite[Lemma 11.129]{HKT} every element of $H$ fixes exactly one place, which is the same place $P$ for all of them. The Riemann-Hurwitz formula \eqref{eq1} applied to the cover $\cX\rightarrow\cX/H$ yields $\bar\gg=0$, $H_P^{(2)}\ne\{1\}$, and
\begin{equation}\label{eqSecondoGruppo} \sum_{i=2}^\infty (|H_P^{(i)}|-1) = 2(q-1)^2.\end{equation}
From \cite[Th. 11.78]{HKT}, $\cX/H_P^{(2)}$ is rational; hence, the Riemann-Hurwitz formula applied to $\cX\to\cX/H_P^{(2)}$ yields
\begin{equation}\label{eqSecondoGruppoBis} \sum_{i=2}^\infty (|H_P^{(i)}|-1) = 2q^2-4q+2|H_P^{(2)}|.\end{equation}
Equations \eqref{eqSecondoGruppo} and \eqref{eqSecondoGruppoBis} provide a contradiction to $H_P^{(2)}\ne\{1\}$.
Suppose $\gamma=1$. Then $H$ is cyclic by Theorem \ref{naka} (IV), a contradiction.

Therefore $\gamma\geq2$. As in Case (A), $\bar\gamma\leq1$ and $k\geq1$; also, if $\bar\gamma=1$, then $k=1$.

Suppose $\bar\gamma=1$ and $k=1$.
From $\gg\geq\gamma$ and the Deuring-Shafarevich formula applied to $\cX\to\bar\cX$ we have $\bar\gg=1$ and $\ell_1\geq2q$; hence, $pq$ divides $\ell_1$. The Riemann-Hurwitz formula applied to $\cX\to\bar\cX$ reads
$$ 2(q-1)^2-2 = q^2(2\cdot0-2)+ \ell_1\sum_{i=0}^{\infty}(|H_P^{(i)}|-1) $$
for any $P$ in $\Omega_1$. This implies that $\ell_1$ divides $q$, a contradiction to $pq\mid\ell_1$.

Therefore $\bar\gamma=0$. Arguing as in the proof of Proposition \ref{prop44} we have $k=2$, $\gamma=q^2+1-(\ell_1+\ell_2)$, and $\bar\gg=0$.
From the Riemann-Hurwitz formula applied to $\cX\to\bar\cX$,
\begin{equation}\label{cesemoquasi}
2(\ell_1+\ell_2)-4q = \ell_1 c_1 + \ell_2 c_2 \geq 0,
\end{equation}
where $c_j:=\sum_{i=2}^\infty (|H_{P_j}^{(i)}|-1)\geq0$ for $j=1,2$.
From Equation \eqref{cesemoquasi}, the integers $\ell_1$ and $\ell_2$ cannot be multiple of $pq$ at the same time. Hence $\ell_1\leq q$ or $\ell_2\leq q$; say $\ell_1\leq q$.
We have $|H_{P_1}^{(2)}|< q^2/\ell_1$ and $|H_{P_2}^{(2)}|< q^2/\ell_2$; otherwise, Equation \eqref{cesemoquasi} would imply $2(\ell_1+\ell_2)-4q\geq q^2-\ell_1$ or $2(\ell_1+\ell_2)-4q\geq q^2-\ell_2$, which is impossible because $\ell_1\leq q$ and $\ell_2\leq q^2/p$.
Therefore, for $j=1,2$, $c_j$ is a multiple of $p$ (possibly zero) from \cite[Lemma 11.75 (v)]{HKT}.
Suppose $\ell_2\geq pq$.
As $c_2\ne2$, Equation \eqref{cesemoquasi} implies that $\ell_2$ divides $[4q+(c_1-2)\ell_1]$; hence, $p$ divides $[2(2q/\ell_1 - 1)]$, a contradiction.
Therefore, $\ell_2\leq q$. Thus, from Equation \eqref{cesemoquasi}, $\ell_1=\ell_2=q$. The rest of the claim follows as in Case (A).
\end{itemize}
\end{proof}

Theorem \ref{vanvan} provides a characterization which generalizes a result by van der Geer and van der Vlugt; see \cite[Proposition 2.2 and Corollary 2.3]{vandervander}.

\begin{theorem} \label{vanvan}
Let $\cY$ be a curve of genus $q-1$ defined over $\K$ whose automorphism group $\aut(\cY)$ contains an elementary abelian subgroup $E_q$ of order $q$.
Then one of the following holds.
\begin{itemize}
\item[(I)] $\cY$ is birationally equivalent over $\K$ to the curve $\cY_{L,a}$ with affine equation
\begin{equation}
\label{vv}
 L(y)=ax+\frac{1}{x},
\end{equation}
for some $a \in \K^*$ and $L(T) \in \K[T]$ a separable $p$-linearized polynomial of degree $q$. For the curve $\cY_{L,a}$ the following properties hold:
\begin{itemize}
\item[(i)] $\cY_{L,a}$ is ordinary and hyperelliptic;
\item[(ii)] $\cY_{L,a}$ has exactly $2q$ Weierstrass places, which are the fixed places of the hyperelliptic involution $\mu$.
\item[(iii)] The full automorphism group $\aut(\cY_L)$ of $\cY_{L,a}$ has order $4q$ and is a direct product $Dih(E_q)\times\langle\mu\rangle$.
\end{itemize}
\item[(II)]  $p\ne3$ and $\cY$ is birationally equivalent over $\K$ to the curve $\cZ_{\tilde L,b}$ with affine equation
\begin{equation}\label{ww}
\tilde L(y)=x^3+bx,
\end{equation}
for some $a\in\K$ and $\tilde L(T)\in\K[T]$ a separable $p$-linearized polynomial of degree $q$.
For the curve $\cZ_{\tilde L,b}$ the following properties hold:
\begin{itemize}
\item[(i)] $\cZ_{\tilde L,b}$ has zero $p$-rank;
\item[(ii)] $\aut(\cZ_{\tilde L,b})$ contains a generalized dihedral subgroup $Dih(E_q)=E_q\rtimes \langle \nu \rangle$.
\end{itemize}
\end{itemize}
\end{theorem}

\begin{proof}
The proof is divided in several steps.
\begin{itemize}
\item
We show that $\cY_{L,a}$ as in \eqref{vv} has genus $q-1$ and $\aut(\cY_{L,a})$ contains a subgroup $Dih(E_q)\times\langle\mu\rangle$.

Let $\bar P_0$ and $\bar P_\infty$ be the zero and pole of $x$ in $\K(x)$, respectively. Then $\K(\cY)|\K(x)$ is a generalized Artin-Schreier extension (\cite[Proposition 3.7.10]{Sti}) which ramifies exactly over the simple poles $\bar P_0$ and $\bar P_\infty$ of $ax+\frac{1}{x}$. Hence, $g(\cY_{L,a})=q-1$. The maps
\begin{equation}\label{autYL}
E_q=\{ \tau_\alpha:(x,y)\mapsto(x,y+\alpha)\mid L(\alpha)=0 \},\quad \nu:(x,y)\mapsto(-x,-y), \quad \mu:(x,y)\mapsto(1/(ax),y),
\end{equation}
generate an automorphism group $Dih(E_q)\times\langle\mu\rangle=(E_q\rtimes\langle\nu\rangle)\times\langle\mu\rangle$ of order $4q$ of $\cY_{L,a}$.

\item We show that $\cY_{L,a}$ is ordinary and hyperelliptic with hyperelliptic involution $\mu$, and that the Weierstrass places of $\cY_{L,a}$ are exactly the $2q$ fixed places of $\mu$.

Let $P_0$ and $P_\infty$ the places of $\cY$ lying over $\bar P_0$ and $\bar P_\infty$. The group $E_q$ and the involution $\nu$ fix $P_0$ and $P_\infty$, while the involution $\mu$ interchanges $P_0$ and $P_\infty$.
Let $\bar\cY=\cY/E_q$ and $\cY^\prime=\cY/\langle\mu\rangle$.
The Riemann-Hurwitz formula applied to the cover $\cY\to\bar\cY$ shows that $\bar\cY$ is rational and $P_0,P_\infty$ are the unique fixed places of any element of $E_q$. Thus, the Deuring-Shafarevich formula applied to $\cY\to\bar\cY$ shows that $\cY$ has $p$-rank $q-1$; hence, $\cY$ is ordinary.
Let $\bar P_1$ and $\bar P_2$ be the distinct zeros of $ax^2+1$ in $\K(x)$, and $P_1^1,\ldots,P_1^q$ and $P_2^1,\ldots,P_2^q$ be the distinct places of $\cY$ lying over $\bar P_1$ and $\bar P_2$. By direct checking, $\mu$ fixes $P_1^1,\ldots,P_1^q,P_2^1,\ldots,P_2^q$. Then the Riemann-Hurwitz formula applied to $\cY\to\cY^\prime$ shows that $\mu$ has no other fixed places and $\cY^\prime$ is rational; hence, $\cY$ is hyperelliptic with hyperelliptic involution $\mu$.
Since $2q>4$, the $2q$ fixed places of $\mu$ are Weierstrass places of $\cY$ from \cite[Theorem 11.112]{HKT}. Moreover, $\cY$ has exactly $2q$ Weierstrass places from \cite[Theorem 7.103]{HKT}.

\item We show that $\cZ_{\tilde L,b}$ as in \eqref{ww} has zero $p$-rank and admits an automorphism group $Dih(E_q)$.

The curve $\cZ_{\tilde L,b}$ admits the automorphism group $Dih(E_q)=E_q\rtimes \langle \nu \rangle$, where
$$ E_q=\{ \tau_\alpha:(x,y)\mapsto(x,y+\alpha)\mid M(\alpha)=0 \},\quad \nu:(x,y)\mapsto(-x,-y). $$
From \cite[Lemma 12.1 (f)]{HKT}, $\cZ_{\tilde L,b}$ has zero $p$-rank.

\item Let $\cY$ be a curve of genus $q-1$ admitting an automorphism group $E_q$ with $\lambda$ fixed places. We show that, if $\lambda=1$, then $p\ne3$ and $\cY$ is birationally equivalent to some $\cZ_{M,b}$.

Let $\bar\cY=\cY/E_q$. The Riemann-Hurwitz formula applied to $\cY\to\bar\cY$ shows that $\bar\cY$ has genus zero and 
\begin{equation}\label{eq19}
2(q-1)=\sum_{i=2}^{\infty} (|(E_{q})_P ^{(i)}|-1) + \sum_{i} \ell_i d_{P_i}, \end{equation}
where $\ell_i$ are the lengths of the short orbits $\Omega_i$ of $E_q$ other than $\{P\}$ and $P_{i}$ is a place of $\Omega_i$; hence, the second summation in Equation \eqref{eq19} is multiple of $p$.
From \cite[Lemma 11.75 (v)]{HKT}, the first summation in \eqref{eq19} is the sum of a multiple of $p$ and $j(q-1)$, where $j$ is the largest integer such that $(E_{q})_P ^{(j+1)} = E_q$.
Thus $j=2$, $E_q=\ldots=(E_q)_P^{(3)}$, $(E_q)_P^{(4)}=\{1\}$, and $\{P\}$ is the unique short orbit of $E_q$.
Let $x\in\K(\bar\cY)$ with $\K(\bar\cY)=\K(x)$ and $\bar P$ be the place of $\bar\cY$ lying under $P$. Up to conjugation in $\aut(\bar\cY)\cong PGL(2,\K)$, $\bar P$ is the simple pole of $x$.
Since $\K(\cY)|\K(x)$ is a generalized Artin-Schreier extension (\cite[Proposition 3.7.10]{Sti}), $\K(\cY)$ is defined as $\K(x,y)$ by $M(y)=h(x)$, where $M(T)\in\K[T]$ is a separable $p$-linearized polynomial of degree $q$ and $h(x)\in\K(x)$.
Since $P$ is the unique ramified place in $\K(x,y)|\K(x)$, Proposition 3.7.10 in \cite{Sti} implies that $h(x)$ is a polynomial function in $\K[x]$ and, in order for the genus of $\cY$ to be $q-1$, the valuation of $x$ at $P$ is $-3$ and coprime to $p$. Hence, $h(T)\in\K[T]$ has degree $3$ and $p\ne3$. Up to a linear transformation in $x$, we can assume that $h(x)$ has the form $x^3+bx+c$; up to a translation in $y$, we can then assume that $c=0$.

\item Let $\cY$ be a curve of genus $q-1$ admitting an automorphism group $E_q$ with $\lambda$ fixed places. We show that, if $\lambda\ne1$, then $\cY$ is birationally equivalent to some $\cY_{L,a}$.

Let $\bar\cY=\cY/E_q$ with genus $\bar\gg$.
From the Riemann-Hurwitz formula applied to $\cY\to\bar\cY$,
\begin{equation}\label{eqbrutta}
 2q-4=q(2\bar\gg-2)+2\lambda(q-1)+\sum_{i=1}^{\lambda}\sum_{j=2}^{\infty}(|(E_q)_{Q_i}^{(j)}|-1) + \sum_{i}\ell_i d_{P_i},
\end{equation}
where $Q_1,\ldots,Q_\lambda$ are the fixed places of $E_q$, $\ell_i$ are the lengths of the short orbits of $E_q$ other than $\{Q_1\},\ldots,\{Q_\lambda\}$, and $P_i$ is a place of the $i$-th short orbit. Note that $\ell_i$ is a multiple of $p$.
If $\lambda=0$, then Equation \eqref{eqbrutta} yields a contradiction modulo $p$.
Then $\lambda\geq2$. Hence, from Equation \eqref{eqbrutta}, $\bar\gg=0$, $\lambda=2$, and $E_q$ has no short orbits other than the two fixed places $P$ and $Q$.
Let $x\in\K(\bar\cY)$ with $\K(\cY)=\K(x)$.
Since $\K(\cY)|\K(x)$ is a generalized Artin-Schreier extension (\cite[Proposition 3.7.10]{Sti}), $\K(\cY)$ is defined as $\K(x,y)$ by $L(y)=h(x)$, for some separable $p$-linearized polynomial $L(T)\in\K[T]$ of degree $q$. Also, from \cite[Proposition 3.7.10]{Sti}, $P$ and $Q$ are the unique poles of $h(x)$, and they are simple poles.
Up to conjugation in $\aut(\bar\cY) \cong PGL(2,\K)$, $\bar P$ and $\bar Q$ are the zero and the pole of $x$.
Therefore, $h(x)=(x-r)(x-s)/x$ for some $r,s\in\K$. Up to formal replacement of $x$ and $y$ with $rsx$ and $y+\delta$, where $\delta\in\K$ satisfies $L(\delta)=-r-s$, the equation $L(y)=h(x)$ is the equation defining the curve $\cY_{L,rs}$.

\item
Finally, we show that $\aut(\cY_{L,a})$ is the group $Dih(E_q)\times\langle\mu\rangle=(E_q\rtimes\langle\nu\rangle)\times\langle\mu\rangle$ described in \eqref{autYL}.

Let $\cY^\prime=\cY/\mu$. Then $\aut(\cY^\prime)$ contains the group $ G^\prime \cong\aut(\cY)/\langle\mu\rangle$ induced by $\aut(\cY)$, and in particular the subgroup $ E_q^\prime \rtimes \langle\nu^\prime \rangle \cong E_q\rtimes\langle\nu\rangle$ induced by $E_q\rtimes\langle\nu\rangle$.
The group $E_q^\prime$ is a Sylow $p$-subgroup of $G^\prime$, because $E_q$ is a Sylow $p$-subgroup of $\aut(\cY)$ from Theorem \ref{naka} (II).
From \cite[Theorem 11.98]{HKT} and \cite[Haptsatz 8.27]{Hup}, either $G^\prime\cong {\rm PSL}(2,q)$, or $G^\prime\cong {\rm PGL}(2,q)$, or $G^\prime = E_q^\prime\rtimes C_m^\prime$, where $C_m^\prime$ is cyclic of order $m$ with $m\mid(q-1)$.

Assume that $G^\prime$ contains a subgroup $E_q^\prime\rtimes C_m^\prime$ with $m\mid(q-1)$. Up to conjugation, $E_q^\prime$ is the group induced by $E_q$ as in \eqref{autYL}. Let $C$ be a tame subgroup of $\aut(\cY)$ inducing $C_m^\prime$. Since $C$ normalizes $E_q$, $C$ acts on the two places of $\cY$ fixed by $E_q$ and acts on the other orbits of $E_q$; since $C$ commutes with $\mu$, $C$ acts on the fixed places of $\mu$, which form two orbits of $E_q$. Thus, the group $\bar C\cong C$ induced by $C$ on the rational curve $\bar\cY=\cY/E_q$ acts on two couples of places. From \cite[Satz 8.5]{Hup}, $\bar C$ has two fixed places and no other short orbits on $\bar\cY$; hence, $\bar C$ has order $2$. This implies $m=2$. For $q-1>2$ the Lemma is then proved, because both ${\rm PGL}(2,q)$ and ${\rm PSL}(2,q)$ contain subgroups $E_q\rtimes C_{q-1}$ of order $q(q-1)$; see \cite[Hauptsatz 8.27]{Hup} and \cite{VM}.

Assume $q=3$. The case $G^\prime\cong{\rm PSL}(2,3)$ is not possible, since ${\rm PSL}(2,3)$ contains no subgroup $Dih(E_3)$.
Suppose $G^\prime\cong{\rm PGL}(2,3)$. Let $\rho^\prime$ be an element of $G^\prime$ of order $4$, and $\rho\in G$ an element of order $4$ inducing $\rho^\prime$.
From \cite[S\"atze 8.2 and 8.4]{Hup} and \cite{VM}, $\rho^\prime$ does not fix the place $P^\prime$ of $\cY^\prime$ lying under the fixed places $P,Q$ of $E_q$. Hence, $P$ and $Q$ are in a long orbit of $\rho$. Therefore, $\rho^\prime$ has a short orbit of length $2$ on $\cY^\prime$. This is impossible, since from \cite[Satz 8.5]{Hup} (see also \cite{VM}) $\rho^\prime$ has two fixed places and no other short orbits on $\cY^\prime$.
We conclude that $G^\prime = E_q^\prime \rtimes C_m^\prime$, and $m=2$ follows as above. The Lemma is thus proved.
\end{itemize}
\end{proof}

\begin{proposition} \label{prop49}
For a curve $\cX$ defined over $\K$, assume that one of the following hold.
\begin{itemize}
\item[(A)] $\cX$ has genus $\gg\leq(q-1)^2$ and $\aut(\cX)$ contains a subgroup $H=E_{q^2}\rtimes(C_2\times C_2)$;
\item[(B)] $\cX$ has genus $\gg=(q-1)^2$ and $\aut(\cX)$ contains a subgroup $H=E_{q^2}$.
\end{itemize}
Then $E_{q^2}$ has a subgroup $T$ of order $q$ such that the quotient curve $\cX /T$ is birationally equivalent over $\K$ to 
the curve $\cY_{L,a}$ in \eqref{vv}, for some $a \in \K^*$ and $L(T) \in \K[T]$ a separable $p$-linearized polynomial of degree $q$.
\end{proposition}

\begin{proof}
From Proposition \eqref{prop44}, $\cX$ is ordinary of genus $(q-1)^2$ and $E_{q^2}$ admits a subgroup $T$ of order $q$ such that the cover $\cX\to\cX/T$ is unramified. From the Riemann-Hurwitz formula and the Deuring-Shafarevich formula applied to $\cX\to\cX/T$, the curve $\cX/T$ is ordinary of genus $q-1$.
Since $T$ is normal in $E_{q^2}$, $\aut(\cX/T)$ contains a subgroup $E_{q^2}/T\cong E_q$.
From Theorem \ref{vanvan}, $\cX/T$ is birationally equivalent over $\K$ to $\cY_{L,a}$ for some $a$ and $L$.
\end{proof}

\begin{proposition} \label{prop410}
Let $\cX$ be a curve admitting an automorphism group $E_{q^2}$ such that, for some $E_q \leq E_{q^2}$ the quotient curve $\cX / E_q$ has affine equation
$$
L(y)=ax+\frac{1}{x},
$$
for some $a \in \K^*$ and $L(T) \in \K[T]$ a separable $p$-linearized polynomial of degree $q$. Then the following hold:
\begin{enumerate}
\item $\K(\cX / E_{q^2})= \K(x).$
\item If $\cX$ is an ordinary curve with genus $(q-1)^2$,  then $E_{q^2}$ contains a subgroup $M$ of order $q$ different from $E_q$ such that the quotient curve $\cX / M$ has affine equation
$$\tilde{L}(z)=b+\frac{1}{x},$$
for some $z \in \K(\cX)$, $b \in \K$, and $\tilde{L}(T) \in \K[T]$ a separable $p$-linearized polynomial of degree $q$.
\end{enumerate}
\end{proposition}

\begin{proof}
Since $[\K(\cX):\K(x)]=q^2=[\K(\cX): \K(\cX / E_{q^2})]$, it is enough to prove that $\tau(x)=x$ for any $\tau \in E_{q^2} \setminus E_q$. Since $\tau$ and $E_q$ commute, $\tau$ induces an automorphism $\tau^\prime$ of $\K(x,y)$. If $\tau^\prime$ is trivial then $\tau(x)=x$ and {\it (1)} follows. Otherwise, $\tau^\prime$ has order $p$. Clearly $E_{q^2} / E_q \cong \tilde{E}_q$, where $\tilde{E}_q$ is an elementary abelian subgroup of $\aut(\cY_L)$ of order $q$. Arguing as in the proof of Theorem \ref{th1}, $\aut(\cY_L)$ has a unique elementary abelian group $F$ of order $q$, namely
$$F=\{\tau_{\alpha}: \ (x,y) \mapsto (x,y+\alpha) \mid L(\alpha)=0\},$$
and hence $F=\tilde{E}_q$. Hence $\tau(x)=x$ for every $\tau \in E_{q^2} \setminus E_q$ and (1) follows.
From {\it (1)}, $\K(\cX / E_{q^2})= \K(x)$, that is, $\cX / E_{q^2}=\mathbb{P}^1(\K)$. The curve $\cY_L$ is the quotient curve $\cX_{(L,L)} / H$, where $$H=\{ \tau_{\alpha,\alpha}: \ (x,y) \mapsto (x+\alpha, y+\alpha) \mid L(\alpha)=0\}.$$ In fact it is sufficient to consider the functions $\eta,\theta \in \K(\cX_{(L,L)})$ with $\eta=L(y)$ and $\theta=x+y$. By direct checking $L(\theta)=\eta+ 1/\eta$ and $\K(\cX_{(L,L)} / H)=\K(\eta,\theta)$. Since $\cX_{(L,L)}$ is an ordinary curve of genus $(q-1)^2$ and the cover $\cX_{(L,L)} \to \cX_{(L,L)}/H$ is unramified, from the Deuring-Shafarevich formula and the Riemann-Hurwitz formula, we have that $\cY_L$ is an ordinary curve of genus $\gg^\prime=q-1$.
The Deuring-Shafarevich formula applied to $E_q$ shows that the extension $\K(\cX) | \K(\cY_L)$ is unramified. Let $P_0$ and $P_\infty$ be respectively the zero and pole of $x$ in $\K(x)$. Then $P_0$ and $P_\infty$ are totally ramified in the extension $\K(\cY_L)|\K(x)$ and no other place of $\mathbb{P}^1(\K)$ ramifies; see \cite[Proposition 3.7.10]{Sti}. Therefore, both $P_0$ and $P_\infty$ split completely in $\cX$. Let $M$ be the stabilizer in $E_{q^2}$ of a place $Q_\infty$ of $\cX$ lying over $P_\infty$. We show that $P_\infty$ is unramified in the extension $\K(\cX/M) | \K(x)$. Note that $|M|=q$, since $P_\infty$ splits in $q$ distinct places in $\cX$.
Furthermore, since $E_{q^2}$ is abelian, each place of $\cX$ lying over $P_\infty$ has the same stabilizer $M$.
Therefore, $P_\infty$ splits completely in $\cX/M$.
Applying the Riemann-Hurwitz formula to the extension $\K(\cX)|\K(\cX /M)$ yields
$$2(q-1)^2-2 \geq q(2\gg(\cX / M)-2) + 2q(q-1).$$
Thus $\gg(\cX /M)=0$. Clearly $[\K(\cX /M):\K(x)]=q$, since
$$q^2=[\K(\cX):\K(x)]=[\K(\cX): \K(\cX /M)][\K(\cX /M): \K(x)]=q[\K(\cX /M): \K(x)].$$
From the Deuring-Shafarevich formula applied to the extension $\K(\cX/M) | \K(x)$, we have that $\K(x)$ has only one place that ramifies in $\K(\cX/M) | \K(x)$, and this place must be $P_0$.

We prove that the quotient curve $\cX/M$ has affine equation
$$\tilde{L}(z)=b+\frac{1}{x},$$
for some $z \in \K(\cX)$, $b \in \K$, and $\tilde{L}(T) \in \K[T]$ a separable $p$-linearized polynomial of degree $q$.
Since $\K(\cX /M)| \K(x)$ is a generalized Artin-Schreier extension (\cite[Proposition 3.7.10]{Sti}), we have that $\K(\cX /M)=\K(x,y)$ where $\tilde{L}(y)=f(x)/g(x)$ for some separable $p$-linearized polynomial $\tilde{L}(T) \in \K[T]$ of degree $q$ and $f(x)/g(x) \in \K(x)$.
Recall that $P_0$ is the unique pole of $f(x)/g(x)$, and it is a simple pole.
\begin{itemize}
\item Suppose that $\deg(f)>\deg(g)$. Then $f(x)/g(x)$ has a pole at $P_\infty$, a contradiction.
\item Suppose that $\deg(f)=\deg(g)>0$. Let $g(x)=x \cdot r(x)^p$ with $r(x) \in \K[x]$, then $f(x)=(x+\alpha)s(x)^p$ with $\alpha\in\K$ and $s(x) \in \K[x]$. If $r(x)$ has a zero $\beta$, then by \cite[Proposition 3.7.10]{Sti} it is easily checked that $f(x)/g(x)$ has a corresponding pole of multiplicity at least $p-1$, a contradiction. Therefore, $g(x)=\beta x$ and $f(x)=x+\alpha$, $\alpha,\beta\in\K$. Applying a linear transformation to $x$, the claim follows.
\item Suppose that $\deg(f)<\deg(g)$ and $\deg(g)>0$. Then, arguing as in the previous case, $f(x)=\alpha $ and $g(x)=\beta x$ with $\alpha,\beta\in\K$. Applying a linear transformation to $x$, the claim follows.
\item Suppose that $\deg(g)=0$. This is impossible since $P_0$ is a pole of $f(x)/g(x)$.
\end{itemize}
\end{proof}

\subsection{Proof of Theorems \ref{th3} and \ref{th4}}
\ \\ \\
We keep our notation introduced in the previous sections.
From Proposition \ref{prop49}, $E_{q^2}$ contains a subgroup $T$ of order $q$ such that the quotient curve $\cX/T$ is the curve $\cY_{L,a}$ with affine equation
$$ L(y)=ax+\frac{1}{x},$$
for some $a \in \K^*$ and $L(T) \in \K[T]$ a separable $p$-linearized polynomial of degree $q$. Let $\K(x,y)$ be the function field $\K(\cX /T)$. From Proposition \ref{prop44}, the $p$-rank of $\cX$ is $\gamma=\gg=(q-1)^2$.
Thus by Proposition \ref{prop410}, $\K(\cX)$ has a subfield $\K(x,z)$ defined by
$$ \tilde{L_1}(z)=b+\frac{1}{x},$$
for some $z \in \K(\cX)$, $b \in \K$, and $\tilde{L_1}(T) \in \K[T]$ a separable $p$-linearized polynomial of degree $q$.
Hence, the compositum $\K(x,y,z)$ of $\K(x,y)$ and $\K(x,z)$ is a subfield of $\K(\cX)$ such that
\begin{equation}
\label{quasi}
\begin{cases} L(y)=ax+\frac{1}{x}, \\
L_1(z)=b+\frac{1}{x}.
\end{cases}
\end{equation}
Therefore, $\K(x,y,z)=\K(y,z)$ with
\begin{equation}
\label{combino}
(L_1(z)-b)L(y)-(L_1(z)-b)^2=a.
\end{equation}
From Proposition \ref{prop410}, $\K(x,z)=\K(\cX)^{M}$ and $\K(x,y)=\K(\cX)^{T}$, where $M \ne T$ is an elementary abelian $p$-subgroup of $E_{q^2}$ of order $q$. Thus,
$$Gal(\K(\cX) \mid \K(y,z))=Gal(\K(\cX) \mid \K(\cX/M)) \cap Gal(\K(\cX) \mid \K(\cX / T))=M \cap T.$$
Since the cover $\cX \to \cX / T$ is unramified, we have $M \cap T = \{1\}$ and hence $\K(\cX)=\K(y,z)$.

\begin{remark}
\label{rem411}
Every $p$-element of $\aut(\cX)$ is an element of $E_{q^2}$.
\end{remark}

\begin{proof}
Let $\sigma$ be a $p$-element of $\aut(\cX)$. By Nakajima's bound, Theorem \ref{naka} (I), $| \langle E_{q^2}, \sigma \rangle | \leq q^2=| E_{q^2} |$. Therefore $\sigma\in E_{q^2}$.
\end{proof}

Let $z^\prime=z-\delta$, with $L_1(\delta)=b$. Then $\K(y,z)=\K(y,z^\prime)$ where
\begin{equation}\label{fine} L_1(z^\prime)L(y)-L_1(z^\prime)^2=a.\end{equation}
Up to a $\K$-scaling of $z^\prime$ and $y$, we can assume that both $L_1$ and $L$ are monic.
Let $\cZ$ be the plane curve with affine equation $L_1(Z^\prime)L(Y)-L_1(Z^\prime)^2=a$.
By Remark \ref{rem411} and Proposition \ref{prop44},
$$E_{q^2}=\{\tau_{\alpha,\beta}:(y,z^\prime)\mapsto(y+\alpha,z^\prime+\beta)\mid L(\alpha)=L_1(\beta)=0\}\leq\aut(\cZ)$$
has exactly two short orbits $\Omega_1$ and $\Omega_2$, which have length $q$ and are centered at the points at infinity $P_1=(1:1:0)$ and $P_2=(1:0:0)$, respectively.
The $q$ distinct tangent lines to $\cZ$ at $P_1$ have equation $\ell_i:Y-Z^\prime=\epsilon_i$, $i=1,\ldots,q$, and the intersection multiplicity at $P_1$ of $\cZ$ and $\ell_i$ is equal to the intersection multiplicity at $P_1$ of the curve $\cW:L(Y)-L_1(Z^\prime)=0$ with the line $\ell_i$. Since $\cW$ has degree $q$, this implies that $\cW$ splits into linear factors $\ell_1,\ell_2,\ldots,\ell_q$.
Therefore $L(Y)-L_1(Z^\prime)=L_2(Y-Z^\prime)$ for some separable $p$-linearized polynomial $L_2(T)\in\K[T]$ of degree $q$.
Thus, Equation \eqref{fine} is the equation (\ref{l1l2}) defining $\cX_{(L_1,L_2)}$, up to the formal replacement of $y-z^\prime$ with $Y$ and of $z^\prime$ with $bX$, where $b^q=a$.

Let $\bar q$ be the largest power of $p$ such that $\aut(\cX)$ contains a cyclic subgroup $C$ of order $\bq-1$.
Up to conjugation in $\aut(\cX)$, $C$ contains the group
$$\Gamma=\{(X,Y)\mapsto(X+\alpha,Y+\beta)\mid L_1(\alpha)=L_2(\beta)=0\}.$$
Then $\cX\in\mathcal{S}_{q|\bq}$ from Theorems \ref{th1} and \ref{th2}. Thus, Theorem \ref{th3} is proved.

If $L_1\ne L_2$, then from Theorem \ref{th2} $\cX_{(L_1,L_2)}$ does not admit any automorphism group $C_2\times C_2$. Thus, also Theorem \ref{th4} is proved.

\section{Acknowledgments}

This research was supported by the Italian Ministry MIUR, Strutture Geometriche, Combinatoria e loro Applicazioni, PRIN 2012 prot. 2012XZE22K, and by GNSAGA of the Italian INdAM.
The authors would like to thank Nazar Arakelian and G\'abor Korchm\'aros for useful comments and suggestions.

\vspace{0.5cm}\noindent {\em Authors' addresses}:

\vspace{0.2cm}\noindent Maria MONTANUCCI\\ Dipartimento di
Matematica, Informatica ed Economia\\ Universit\`a degli Studi  della Basilicata\\ Contrada Macchia
Romana\\ 85100 Potenza (Italy).\\E--mail: {\tt
maria.montanucci@unibas.it}

\vspace{0.2cm}\noindent Giovanni ZINI \\ Dipartimento di
Matematica e Informatica \\ Universit\`a degli Studi  di Firenze \\ Viale Morgagni \\ 50134 Firenze (Italy).\\E--mail: {\tt
gzini@math.unifi.it}


\begin{thebibliography}{30}
\bibitem{AK} N. Arakelian and G. Korchm\'aros, A characterization of the Artin-Mumford curve, \textit{J. Number Theory}, \textbf{154} (2015), 278-291.
\bibitem{CKK} G. Cornelissen and F. Kato, Equivariant deformation of Mumford curves and of ordinary curves in positive characteristic, \textit{Duke Math. J}, \textbf{166}, (2003), 431-470.
\bibitem{CK} G. Cornelissen and F. Kato, Discontinuous groups in positive characteristic and automorphisms of Mumford curves, \textit{Math. Ann.}, \textbf{320}, (2001), 55-85.
\bibitem{CK2} G. Cornelissen and F. Kato, Mumford curves with maximal automorphism group, \textit{Proceedings of the American Mathematical Society}, \textbf{132}, (2004), 1937-1941.
 \bibitem{GK2016} M.~Giulietti and G.~Korchm\'aros, Algebraic curves with many automorphisms, preprint (2016).
 \bibitem{HKT} J.W.P. Hirschfeld, G. Korchm\'aros, and F. Torres, {\it Algebraic Curves over a Finite Field}, Princeton Series in Applied Mathematics, Princeton (2008).
\bibitem{Hup} B.~Huppert, \emph{ Endliche Gruppen. I}, Grundlehren der Mathematischen
wissenschaften {\bf 134}, Springer, Berlin, 1967, xii+793 pp.
\bibitem{KM} G.~Korchm\'aros and M.~Montanucci, Ordinary algebraic curves with many automorphisms in positive
characteristic, arXiv:1610.05252, 2016.
\bibitem{KR} A.~Kontogeorgis and V.~Rotger, On abelian automorphism groups of Mumford curves, \emph{Bull. London Math. Soc.} \textbf{40} (2008), 353-362.
\bibitem{N} S.~Nakajima, $p$-ranks and automorphism groups of algebraic curves, \textit{Trans. Amer. Math. Soc.} \textbf{303} (1987), 595-607.
\bibitem{Sti} H. Stichtenoth, {\it Algebraic function fields and codes}, 2nd edn. Graduate Texts in Mathematics {\bf 254}. Springer, Berlin (2009).
\bibitem{VM} R.Valentini and M- Madan, A Hauptsatz of L.E. Dickson and Artin-Schreier extensions, \textit{Journal f\"ur die reine und angewandte Mathematik}, (1980), 156-177.
\bibitem{vandervander} G. van der Geer and M. van der Vlugt, Kloosterman sums and the $p$-torsion of certain Jacobians, \textit{Math. Ann.} \textbf{290}, Birkh\"auser, Basel, (1991), 549-563.
\end{thebibliography}
    \end{document}